\documentclass[a4paper]{amsart}
\usepackage{url}
\usepackage[all]{xy}
\usepackage[mathscr]{eucal}
\usepackage{amsmath,amssymb,amsfonts}
\usepackage{mathrsfs,latexsym,amsthm,enumerate}
\newtheorem{theorem}{Theorem}[section]
\newtheorem{lemma}[theorem]{Lemma}
\newtheorem{corollary}[theorem]{Corollary}
\newtheorem{example}[theorem]{Example}

\newtheorem{proposition}[theorem]{Proposition}
\newtheorem{remark}[theorem]{Remark}

\usepackage{xcolor}

\title[The Cuntz-Krieger relations]{The universal Boolean inverse semigroup presented by the abstract Cuntz-Krieger relations}

\author{Mark V. Lawson}
\address{Mark V. Lawson, Department of Mathematics
and the
Maxwell Institute for Mathematical Sciences, 
Heriot-Watt University,
Riccarton,
Edinburgh EH14 4AS, 
UNITED KINGDOM}
\email{m.v.lawson@hw.ac.uk}

\author{Alina Vdovina}
\address{Alina Vdovina, 
School of Mathematics, Statistics and Physics,
Herschel Building,
University of Newcastle,
Newcastle-upon-Tyne NE1 7RU,
UNITED KINGDOM}
\email{alina.vdovina@ncl.ac.uk}

\begin{document} 

\begin{abstract} 
This paper is a contribution to the theory of what might be termed $0$-dimensional non-commutative spaces.
We prove that associated with each inverse semigroup $S$ is a Boolean inverse semigroup presented by the abstract versions of the Cuntz-Krieger relations.
We call this Boolean inverse semigroup the Exel completion of $S$ and show that it arises from Exel's tight groupoid under non-commutative Stone duality.
\end{abstract}
\maketitle

\section{Introduction}

In this section, we explain the philosophy behind this paper and provide the context for the two theorems (Theorem~\ref{them:three} and Theorem~\ref{them:four}) that we prove;
any undefined terms will be defined later in this paper. 

The theory of $C^{\ast}$-algebras is the theory of non-commutative spaces.
The term `non-commutative space' is mathematical legerdemain --- there is no actual space
in the background, unlike in the case of commutative $C^{\ast}$-algebras;
instead, the $C^{\ast}$-algebra is itself a proxy for what is absent.
For some $C^{\ast}$-algebras, however, there is an honest-to-goodness space, to be regarded as an actual non-commu\-tative space, from which they are constructed.
These are the \'etale groupoid $C^{\ast}$-algebras of Renault \cite{Renault}
which include amongst their number many interesting and important examples \cite{Renault, K, Paterson, Kellendonk1, Kellendonk2, Exel}.
It is often the case that the \'etale groupoids that occur in constructing such $C^{\ast}$-algebras
are those  whose spaces of identities are locally compact Boolean spaces
---
by which we mean $0$-dimensional, locally compact Hausdorff spaces.
A prime example of such a space, and one which occurs repeatedly in the theory of $C^{\ast}$-algebras,
is the Cantor space.
Thus locally compact Boolean spaces are natural generalizations of the Cantor space.
Define a Boolean groupoid to be an \'etale groupoid whose space of identities is a locally compact Boolean space.
Boolean groupoids are therefore examples of what can be regarded as (concrete) $0$-dimensional, non-commutative spaces.

Most of the time, Boolean groupoids are studied {\em tout court} but, in fact, they have algebraic {\em doppelg\"angers}.
It is a classical theorem due to Marshall Stone \cite{Stone1936, Stone1937a, Stone1937b} (and sketched in Section~3) that locally compact Boolean spaces stand in duality to generalized Boolean algebras:
from a generalized Boolean algebra, a locally compact Boolean space, called its Stone space, can be constructed from its set of ultrafilters,
and from a locally compact Boolean space, a generalized Boolean algebra can be constructed whose elements are the compact-open sets of the space.
This classical duality, which can be viewed as being commutative in nature, has been generalized to
a non-commutative setting \cite{Resende2, Law3, Law5, LL, KL, LMS} (and sketched in Section~4):
locally compact Boolean spaces are replaced by Boolean groupoids, and generalized Boolean algebras by what we call Boolean inverse semigroups.
Just as in the classical case, from a Boolean inverse semigroup, a Boolean groupoid, called its Stone groupoid, can be constructed from its set of ultrafilters
and from a Boolean groupoid, a Boolean inverse semigroup can be constructed whose elements are the compact-open partial bisections.
This result suggests two lines of research: 
\begin{enumerate}
\item Develop the theory of Boolean inverse semigroups as the non-commutative theory of Boolean algebras.
\item Reinterpret results about Boolean groupoids as results about Boolean inverse semigroups (and vice versa).
\end{enumerate}
The starting point for this paper are two theorems that belong, respectively, to precisely these two lines of research.
The first is a theorem \cite{LL,Lawson2018} which generalizes a well-known result in the theory of Boolean algebras:
namely, that associated with every distributive lattice is a universal Boolean algebra into which it may be embedded \cite{J}.\\

\noindent
{\bf Terminology. }The inverse semigroups in this paper will always have a zero and homomorphisms between them will always be required to preserve it.
In addition, homomorphisms between monoids will always be required to map identities to identities.
If we say `semigroup' we mean that we do not assume there is an identity.
We shall use the term `Boolean algebra' rather than `generalized Boolean algebra'
and `unital Boolean algebra' for what is usually termed a `Boolean algebra'.
In particular, a `Boolean inverse semigroup' will therefore have a semilattice of idempotents which is a generalized Boolean algebra ---
we do not assume it has an identity.\\

\begin{theorem}[Booleanization]\label{them:one} From each inverse semigroup $S$, we may construct a Boolean inverse semigroup $\mathsf{B}(S)$,
called its {\em Booleanization}, together with a (semigroup) homomorphism $\beta \colon S \rightarrow \mathsf{B}(S)$ which is universal for
homomorphisms from $S$ to Boolean inverse semigroups;
this means precisely that if $\theta \colon S \rightarrow T$ is any homomorphism to a Boolean inverse semigroup $T$, then there is a unique morphism 
$\theta' \colon \mathsf{B}(S) \rightarrow T$ of Boolean inverse semigroups such that $\theta' \beta = \theta$.
\end{theorem}

The second theorem then answers the question of what the Stone groupoid of the Booleanization is \cite{LL,Lawson2018}.

\begin{theorem}[The universal groupoid]\label{them:two} The Stone groupoid of the Booleanization $\mathsf{B}(S)$
is Paterson's {\em universal groupoid} $\mathsf{G}_{u}(S)$.
\end{theorem}

Paterson's universal groupoid is described in his book \cite{Paterson}.
In fact, his construction came first and it was as a result of thinking about what he was doing that the above theorem came to be proved.
This, then, is the conceptual background to our paper.
We can now turn to the two particular results that we prove here;
they will also exemplify the two lines of research mentioned above
and each can be seen as a specialization of the above two theorems.

The papers of Cuntz and Krieger \cite{Cuntz, CK} led to the idea of building $C^{\ast}$-algebras from combinatorial structures.
Central to this work has been the presentation of certain $C^{\ast}$-algebras by means of `Cuntz-Krieger relations'.
The goal of our paper can now be explicitly stated: it is to describe in abstract terms exactly what these relations are.
There are two new results.

Our first new result is an application of Theorem~\ref{them:one} and uses the theory of ideals of Boolean inverse semigroups
described in \cite{Wehrung}. It is based on two ideas: that of a cover of an element and that of a cover-to-join map.
(In fact, covers and cover-to-join maps are important features of frame theory \cite{J} whereas non-commutative Stone duality
can be regarded as part of non-commutative frame theory.)
The notion of a cover was developed in a sequence of papers \cite{Lawson2010,Law5,LL} but was rooted in the seminal papers by Exel \cite{Exel} and Lenz \cite{Lenz}.
A subset $\{a_{1}, \ldots, a_{m}\}$ of the principal order ideal generated by the element $a$ is a {\em cover} of $a$
if for each $0 \neq x \leq a$ there exists $1 \leq i \leq m$ such that $x \wedge a_{i} \neq 0$.
(As an aside, observe that in an inverse semigroup, compatible elements have meets \cite[Lemma~1.4.11]{Law1} and all the elements of a principal order ideal are compatible.)\\

\noindent
{\bf Terminology. }Our use of the word `cover' is a special case of the way this word is used in \cite{Exel}.
Observe that we only use covers that are contained in principal order ideals.\\

The notion of a cover in an arbitrary inverse semigroup is a weakening of the notion of a join.
A cover-to-join map from an inverse semigroup to a Boolean inverse semigroup converts covers to joins: thus, it converts such potential joins to actual joins.
It is the claim of this paper that covers {\em are} the abstract form of the concrete Cuntz-Krieger relations that arise in particular examples.
This claim will be justified in Section~11.
The inverse semigroup $S$ is embedded in its Booleanization $\mathsf{B}(S)$ so we may identify $S$ with its image.
Let $\{a_{1}, \ldots, a_{m}\}$ be a cover of $a$.
Then, in particular, $\{a_{1}, \ldots, a_{m}\}$ is a compatible set in $S$ and so will have a join in $\mathsf{B}(S)$.
{\em Inside} $\mathsf{B}(S)$, we of course have that $a_{1} \vee \ldots \vee a_{m} \leq a$.
It follows that the element $a \setminus (a_{1} \vee \ldots \vee a_{m})$ is defined in  $\mathsf{B}(S)$.
Let $I$ be the additive ideal 
of  $\mathsf{B}(S)$ generated by these elements.
We call $I$ the {\em Cuntz-Krieger ideal} of $\mathsf{B}(S)$.
Put $\mathsf{T}(S) = \mathsf{B}(S)/I$ and let $\tau \colon S \rightarrow \mathsf{T}(S)$ be the natural map. 
We call $\mathsf{T}(S)$ the {\em Exel completion} of $S$.

\begin{theorem}[Exel completion]\label{them:three} Let $S$ be an inverse semigroup.
Then $\tau \colon S \rightarrow \mathsf{T}(S)$ is a cover-to-join map which is universal for all cover-to-join
maps from $S$ to Boolean inverse semigroups;
this means precisely that for each cover-to-join map $\theta \colon S \rightarrow T$ to a Boolean inverse semigroup $T$ there is a unique
morphism $\theta' \colon \mathsf{T}(S) \rightarrow T$ of Boolean inverse semigroups such that $\theta' \tau = \theta$.
\end{theorem}

The Exel completion of an inverse semigroup $S$ should be regarded as the Boolean inverse semigroup generated by $S$ subject to the
abstract Cuntz-Krieger relations.
Our second new result, which is the main theorem of this paper, is a description of the Stone groupoid of the Exel completion of $S$.
This involves what is termed the tight groupoid $\mathsf{G}_{t}(S)$ of an inverse semigroup $S$, introduced in \cite{Exel};
it will be explicitly defined at the beginning of Section~9.

\begin{theorem}\label{them:four}  Let $S$ be an inverse semigroup.
Then the Stone groupoid of the Exel completion $\mathsf{T}(S)$ of $S$ is the tight groupoid $\mathsf{G}_{t}(S)$.
\end{theorem}

\noindent
{\bf Acknowledgements. }The germ of the work described in this paper arose from a discussion on higher rank graphs amongst Nadia Larsen, Mark V. Lawson, Aidan Sims and Alina Vdovina at the end of the
2017 ICMS Workshop {\em Operator algebras: order, disorder and symmetry}.
This led to ongoing discussions between the two authors centred on the papers \cite{RS,KP,RSY,FMY,Sp1,Sp2}. 
It quickly became clear that there was a need to find a common language and the present paper was the result.
Crucial to our thinking, was the work of Ruy Exel \cite{Exel} and Daniel Lenz \cite{Lenz};
our main theorem (Theorem~\ref{them:four}) is analogous to a result of Benjamin Steinberg \cite[Corollary 5.3]{Steinberg2016}
but we work, of course, with Boolean inverse semigroups.
Our use of covers and cover-to-join maps goes back to \cite{LL} although they also play a role in \cite{Exel};
see also \cite{DM}, a paper tightly linked to this one;
in particular, the authors would like to thank Allan Donsig for answering some of their questions.
In the first version of this paper, the authors proved Theorem~\ref{them:four} under the assumption that the inverse semigroup was a 
`weak semilattice' in the sense of Steinberg \cite{Steinberg, Steinberg2016}.
The authors would like to thank Enrique Pardo for pointing out that this was unnecessary and supplying a result,
suggested by Lisa Orloff Clark, that enabled us to prove the more general version of the theorem.
Finally, the authors would like to thank Ruy Exel for a very constructive email exchange.
As a result, we would like to point out that the tight maps defined in \cite{Exel} are of a more general nature than the maps we use.
This is because Exel does not want to spell out whether he is working in a unital or non-unital environment;
it is a feature of our categorical setting, that we have to be explicit.

\section{Inverse semigroups and groupoids}

We assume the reader is familiar with basic inverse semigroup theory \cite{Law1} and that of \'etale groupoids \cite{Resende}.\\
If $s$ is an element of an inverse semigroup we write $\mathbf{d}(s) = s^{-1}s$ and $\mathbf{r}(s) = ss^{-1}$.
We write $e \stackrel{a}{\rightarrow} f$ to mean that $\mathbf{d}(a) = e$ and $\mathbf{r}(a) = f$.
Green's relation $\mathscr{D}$ assumes the following form in inverse semigroups:
$a\, \mathscr{D} \, b$ if and only if there is an element $x$ such that $\mathbf{d}(a) \stackrel{x}{\rightarrow} \mathbf{d}(b)$.
The order on inverse semigroups will be the usual {\em natural partial order}.
The semilattice of idempotents of an inverse semigroup $S$ is denoted by $\mathsf{E}(S)$.
More generally, if $X$ is a subset of $S$ then $\mathsf{E}(X) = \mathsf{E}(S) \cap X$.
In addition, define
$$X^{\uparrow} = \{s \in S \colon \exists x \in X, x \leq s \}
\text{ and }
X^{\downarrow} = \{s \in S \colon \exists x \in X, s \leq x \}.$$
If $X = \{x\}$ then we write simply $x^{\uparrow}$ and $x^{\downarrow}$, respectively.
The {\em compatibility relation} $\sim$ in an inverse semigroup  
is defined by $s \sim t$ if and only if $s^{-1}t$ and $st^{-1}$ are idempotents.
The significance of the compatibility relation is that being compatible is a necessary condition for two elements to have a join.
A set that consists of elements which are pairwise compatible is said to be {\em compatible}.
The {\em orthogonality relation} $\perp$ in an inverse semigroup 
is defined by $s \perp t$ if and only if $s^{-1}t = 0 = st^{-1}$.
A set that consists of elements which are pairwise orthogonal is said to be {\em orthogonal}.

If $G$ is a groupoid we regard it as a set of arrows.
Amongst those arrows are the identities and the set of such identities is denoted by $G_{o}$.
If $g \in G$ we write $\mathbf{d}(g) = g^{-1}g$ and $\mathbf{r}(g) = gg^{-1}$.
We write $e \stackrel{g}{\rightarrow} f$ if $\mathbf{d}(g) = e$ and $\mathbf{r}(g) = f$.
Define the equivalence relation $\mathcal{D}$ on $G$ by $g \, \mathcal{D} \, h$ if and only if
there exists $x \in G$ such that $\mathbf{d}(g) \stackrel{x}{\rightarrow} \mathbf{d}(h)$.
A subset of $G$ is said to be an {\em invariant} subset if it is a union of $\mathcal{D}$-classes.
A subset of $G_{o}$ is said to be an {\em invariant} subset if it is a union of $\mathcal{D}$-classes restricted to $G_{o}$.
Observe that a subset $X$ of $G_{o}$ is invariant precisely when it satisfies the following condition: $g^{-1}g \in X \Leftrightarrow gg^{-1} \in X$.
Let $G$ be a groupoid and let $X \subseteq G_{o}$ be any subset of the space of identities.
The {\em reduction of $G$ to $X$}, denoted by $G|_{X}$, is the groupoid whose elements are all those $g \in G$ such that
$\mathbf{d}(g), \mathbf{r}(g) \in X$.
A functor $\alpha \colon G \rightarrow H$ is said to be a {\em covering} functor if for each identity $e \in G$ 
the induced function from the set $\{g \in G \colon \mathbf{d}(g) = e\}$ to the set  $\{h \in H \colon \mathbf{d}(h) = \alpha (e)\}$
is a bijection.
Let $G$ be any groupoid.
A subset $X \subseteq G$ is said to be a {\em partial bisection} if $x,y \in X$ and $\mathbf{d}(x) = \mathbf{d}(y)$ then $x = y$,
and if $x,y \in X$ and $\mathbf{r}(x) = \mathbf{r}(y)$ then $x = y$.
This is equivalent to requiring that $X^{-1}X,XX^{-1} \subseteq G_{o}$.
In this paper, we are interested in topological groupoids,
that is groupoids which carry a topology with respect to which multiplication and inversion are continuous,
but more specifically those topological groupoids which are also {\em \'etale}, meaning that the domain
and range maps are local homeomorphisms.

\section{Commutative Stone duality}

Classical Stone duality \cite{Stone1936, Stone1937a, Stone1937b} is described in the book \cite{J} where it is unfortunately limited to the unital case.
We therefore sketch out the essentials we shall need of the non-unital theory here.
Distributive lattices will always have a bottom but not necessarily a top.
A {\em generalized Boolean algebra} is then a distributive lattice with bottom element in which each principal order ideal is a unital Boolean algebra.
In a distributive lattice, every ultrafilter is a prime filter \cite[Theorem 3]{Stone1937b}
and a distributive lattice is a generalized Boolean algebra if and only if every prime filter is an ultrafilter \cite[Proposition 1.6]{LL}.

Let $X$ be a Hausdorff space.
Then $X$ is {\em locally compact} if each point of $X$ is contained in the interior of a compact subset \cite[Theorem 18.2]{Willard}.
Recall that a topological space is {\em $0$-dimensional} if it has a basis of clopen subsets.
The proof of the following is by standard results in topology \cite{Simmons}.
It is included solely to provdie context.

\begin{lemma}\label{lem:boolean-space} Let $X$ be a Hausdorff space.
Then the following are equivalent.
\begin{enumerate}
\item $X$ is locally compact and $0$-dimensional.
\item $X$ has a basis of compact-open sets.
\end{enumerate}
\end{lemma}

We define a {\em locally compact Boolean space} to be a $0$-dimensional, locally compact Hausdorff space
and a {\em compact Boolean space} to be a $0$-dimensional, compact Hausdorff space.
Let $B_{1}$ and $B_{2}$ be Boolean algebras.
A morphism $\alpha \colon B_{1} \rightarrow B_{2}$ of such algebras is said to be {\em proper}
if $B_{2} = \mbox{im}(\alpha)^{\downarrow}$.
Let $X_{1}$ and $X_{2}$ be locally compact Boolean spaces.
A continuous map $\beta \colon X_{2} \rightarrow X_{1}$ is said to be {\em proper}
if the inverse image under $\beta$ of each compact set is compact.

\begin{theorem}[Commutative Stone duality]\label{them:csd2} 
The category of Boolean algebras (respectively, unital Boolean algebras) and their proper morphisms (respectively, morphisms)
is dually equivalent to the category of locally compact Boolean spaces (respectively, compact Boolean spaces) and their proper morphisms (respectively, continuous maps).
\end{theorem}

\section{Non-commutative Stone duality}

We refer the reader to the papers \cite{Law3, Law5, LL} for all the details omitted in this section.
An inverse semigroup is said to be {\em distributive} if it has binary joins of compatible elements and multiplication distributes over such joins.
A distributive inverse semigroup is {\em Boolean} if its semilattice of idempotents is a Boolean algebra.
If $X \subseteq S$ is a subset of a distributive inverse semigroup, denote by $X^{\vee}$ the set of all joins
of finite, non-empty compatible subsets of $S$.
Clearly, $X \subseteq X^{\vee}$. 
A {\em morphism} between distributive inverse semigroups is a homomorphism of inverse semigroups
that maps binary compatible joins to binary compatible joins.

Let $S$ be an inverse semigroup.
A {\em filter} in $S$ is a subset $A$ such that $A = A^{\uparrow}$ and whenever $a,b \in A$ there exists $c \in A$ such that $c \leq a,b$.
A filter is {\em proper} if it does not contain zero.\\

\noindent
{\bf Terminology. }Proper filters are always assume to be non-empty.\\

Observe that $A$ is a filter if and only if $A^{-1}$ is a filter.
If $A$ and $B$ are filters then $(AB)^{\uparrow}$ is a filter.
Define $\mathbf{d}(A) = (A^{-1}A)^{\uparrow}$ and $\mathbf{r}(A) = (AA^{-1})^{\uparrow}$.
Then both $\mathbf{d}(A)$ and $\mathbf{r}(A)$ are filters.
It is easy to check that $A$ is proper if and only if $\mathbf{d}(A)$ is proper (respectively, $\mathbf{r}(A)$ is proper).
Observe that for each $a \in A$ we have that $A = (a \mathbf{d}(A))^{\uparrow} = (\mathbf{r}(A)a)^{\uparrow}$.
We denote the set of {\em proper} filters on $S$ by $\mathcal{L}(S)$.
If $A,B \in \mathcal{L}(S)$, then $A \cdot B$ is defined if and only if $\mathbf{d}(A) = \mathbf{r}(B)$
in which case $A \cdot B = (AB)^{\uparrow}$.
In this way, $\mathcal{L}(S)$ becomes a groupoid;
the identities of this groupoid are the filters that contain idempotents --- these are precisely the filters that are also inverse subsemigroups. 

\begin{remark}{\em Let $E$ be a meet semilattice with zero.
Then proper filters (recall that they are always required to be non-empty) on $E$ correspond exactly to the characters of Exel \cite[page 3, page 40, page 53]{Exel}.
However, proper filters can be extended to arbitrary inverse semigroups and form the basis of the approach to non-commutative
Stone duality developed in this paper.
This approach goes back to the paper of Lenz \cite{Lenz} as developed in \cite{LMS}.
In addition, the term `character' has other meanings in algebra and so is one that has to be used with caution.}
\end{remark}

Let $S$ be a distributive inverse semigroup.
A {\em prime filter} in $S$ is a proper filter $A \subseteq S$ such that if $a \vee b \in A$
then $a \in A$ or $b \in A$.
An {\em ultrafilter} is a maximal proper filter.
Denote the set of all prime filters of $S$ by $\mathsf{G}(S)$.
It can be checked that $A$ is a prime filter if and only if  $\mathbf{d}(A)$ (respectively, $\mathbf{r}(A)$) is a prime filter.
Define a partial multiplication $\cdot$ on $\mathsf{G}(S)$ by $A \cdot B$ exists if and only if $\mathbf{d}(A) = \mathbf{r}(B)$,
in which case $A \cdot B = (AB)^{\uparrow}$.
With respect to this partial multiplication, $\mathsf{G}(S)$ is a groupoid;
the identities are the prime filters that contain idempotents.
For this reason, it is convenient to define a prime filter to be an {\em identity} if it contains an idempotent.
Proofs of all of the above claims can be found in \cite{LL}.
In a distributive inverse semigroup all ultrafilters are prime filters 
whereas Boolean inverse semigroups are characterized by the fact that all prime filters are ultrafilters \cite[Lemma~3.20]{LL}.

An \'etale groupoid $G$ is called a {\em Boolean groupoid} if its space of identities is a locally compact Boolean space.
Let $S$ be a Boolean inverse semigroup.
Denote by $\mathsf{G}(S)$ the set of all ultrafilters of $S$.
Then $\mathsf{G}(S)$ is a Boolean groupoid, called the {\em Stone groupoid} of $S$, 
where a basis for the topology is given by the subsets $V_{a}$, the set of all ultrafilters in $S$ that contain the element $a \in S$.
Let $G$ be a Boolean groupoid.
Denote by $\mathsf{KB}(G)$ the set of all 
compact-open partial bisections of $G$.
Then $\mathsf{KB}(G)$ is a Boolean inverse semigroup under subset multiplication.
A morphism $\theta \colon S \rightarrow T$ between Boolean inverse semigroups is said to be {\em callitic} if it satisfies two properties:
\begin{enumerate}
\item It is {\em weakly meet preserving} meaning that for any $a,b  \in S$ and any $t \in T$  if $t \leq \theta (a), \theta (b)$ then there exists $c \leq a,b$ such that $t \leq \theta (c)$.
\item It is {\em proper} meaning that $\mbox{im}(\theta)^{\vee} = T$. Observe that surjective maps are automatically proper.
\end{enumerate}
A continuous functor $\alpha \colon G \rightarrow H$ between \'etale groupoids is said to be {\em coherent} if
the inverse images of compact-open sets are compact-open.
The following is the non-commutative generalization of Theorem~\ref{them:csd2}.

\begin{theorem}[Non-commutative Stone duality]\label{them:ncsd} \mbox{}
\begin{enumerate}

\item For each Boolean inverse semigroup $S$, the group\-oid $\mathsf{G}(S)$ is Boolean and is such that $S \cong \mathsf{KB}(\mathsf{G}(S))$.

\item For each Boolean groupoid $G$, the semigroup $\mathsf{KB}(G)$ is a Boolean inverse semigroup and is such that $G \cong \mathsf{G}(\mathsf{KB}(G))$.

\item There is a dual equivalence between callitic morphisms and coherent continuous covering functors.

\end{enumerate}
\end{theorem}


\section{Additive ideals}

This section contains those results about Boolean inverse semigroups that are `ring-like'.
Specifically, Proposition~\ref{prop:earth} will be the key to proving Theorem~\ref{them:four}.
It will require a refinement of some of the results proved in \cite{Wehrung}.\\

\noindent
{\bf Terminology.} In the theory of Boolean inverse semigroups, there are two notions of `kernel'.
The first, which we shall write as {\em Kernel}, is the congruence induced by a morphism on its domain.
The second, which we shall write as {\em kernel}, is the set of all elements of the domain sent to zero.
The congruences induced on the domains of morphisms are called {\em additive congruences}.
The use of the word `additive' arises from regarding the partially defined binary operation of compatible join as an
analogue of addition in rings. Wehrung provides an abstract characterization of additive congruences in \cite[Proposition~3.4.1]{Wehrung}
but we shall only need the informal idea here.\\

The fundamental problem in working with Boolean inverse semigroups is that joins are only defined for compatible subsets.
Wehrung \cite[Section~3.2]{Wehrung} devised an ingenious solution to deal with this issue that enabled him to show that, despite appearances,
Boolean inverse semigroups form a variety of algebras.
Let $a,b \in S$, a Boolean inverse semigroup.
Put 
$e = \mathbf{d}(a) \setminus \mathbf{d}(a)\mathbf{d}(b)$
and
$f = \mathbf{r}(b) \setminus \mathbf{r}(a) \mathbf{r}(b)$.
Define 
$$a \ominus b = fae.$$
This is called the {\em (left) skew difference}.
The element $a \ominus b$ is the largest element of $a^{\downarrow}$ orthogonal to $b$.
Define 
$$a \, \triangledown \, b = (a \ominus b) \vee b.$$
This is called the {\em (left) skew join of $a$ and $b$}.
The important point about the left skew join is that it is always defined
and, as we show next, extends the partially defined operation of binary compatible join.

\begin{lemma}\label{lem:fanta} Let $S$ be a Boolean inverse semigroup.
If $s \sim t$ then $s \, \triangledown \, t = s \vee t$.
\end{lemma} 
\begin{proof} If $s \sim t$ then $s \wedge t$ exists 
and $\mathbf{d}(s \wedge t) = \mathbf{d}(s) \wedge \mathbf{d}(t)$ 
and
$\mathbf{r}(s \wedge t) = \mathbf{r}(s) \wedge \mathbf{r}(t)$ by \cite[Lemma~1.4.11]{Law1}.
It follows that $s \ominus t = s \setminus (s \wedge t)$.
Thus  $s \, \triangledown \, t = s \vee t$, as claimed.
\end{proof}

Skew join is an algebraic operation and is preserved by all morphisms between Boolean inverse semigroups.
The following result is simple, but useful.

\begin{lemma}\label{lem:irn} Let $\theta \colon S \rightarrow T$ be a morphism of Boolean inverse semigroups.
If $\theta (a) \sim \theta (b)$ then $\theta (a) \vee \theta (b) = \theta (a \, \triangledown \, b)$.
\end{lemma}
\begin{proof} The element $a \, \triangledown \, b$ exists in $S$ and $\theta (a \, \triangledown \, b) = \theta (a) \, \triangledown \, \theta (b)$.
But by Lemma~\ref{lem:fanta} and the assumption that $\theta (a) \sim \theta (b)$ we get that $\theta (a \, \triangledown \, b) = \theta (a) \vee \theta (b)$.
\end{proof}

Let $S$ be a Boolean inverse semigroup.
A (semigroup) ideal $I$ of $S$ is said to be {\em additive} if it is closed under binary compatible joins.
Recall that if $X \subseteq S$ then $X^{\vee}$ denotes the set of all finite joins of non-empty compatible subsets of $X$.
The proof of the following is routine.

\begin{lemma}\label{lem:pop} Let $S$ be a Boolean inverse semigroup and let $X \subseteq S$.
Then $(SXS)^{\vee}$ is the smallest additive ideal in $S$ containing $X$.
\end{lemma}

Additive ideals arise from morphisms between Boolean inverse semigroups.
Let $\theta \colon S \rightarrow T$ be a morphism between Boolean inverse semigroups.
The set 
$$\mbox{ker}(\theta) = \{s \in S \colon \theta (s) = 0\}$$
is called the {\em kernel} of $\theta$.
Clearly, $\mbox{ker}(\theta)$ is an additive ideal of $S$. 
Similarly, we define the kernel of an additive congruence to be the class of the zero.
However, Boolean inverse semigroups are not rings and not every morphism is determined by its kernel.
We now examine which are.
Let $I$ be an additive ideal of the Boolean inverse semigroup $S$.
Define the relation $\varepsilon_{I}$ on $S$ as follows:
$$(a,b) \in \varepsilon_{I}
\Leftrightarrow
\exists
c \leq a,b \text{ such that } (a \setminus c), (b \setminus c) \in I.$$
Then $\varepsilon_{I}$ is an additive congruence with kernel $I$.
We shall write $S/I$ instead of $S/\varepsilon_{I}$.
We say that an additive congruence is {\em ideal-induced} if it equals $\varepsilon_{I}$ for some additive ideal $I$.
The following result is due to Ganna Kudryavtseva (private communication)
and characterizes exactly which morphisms of Boolean inverse semigroups are ideal-induced.





\begin{proposition}\label{prop:anja} A morphism of Boolean inverse semigroups is weakly meet preserving if
and only if its associated congruence is ideal-induced.
\end{proposition}
\begin{proof} Let $I$ be an additive ideal of $S$ and let $\varepsilon_{I}$ be its associated additive congruence on $S$.
Denote by $\nu \colon S \rightarrow S/\varepsilon_{I}$ is associated natural morphism.
We prove that $\nu$ is weakly meet preserving.
Denote the $\varepsilon_{I}$-class containing $s$ by $[s]$.
Let $[t] \leq [a], [b]$.
Then $[t] = [at^{-1}t]$ and $[t] = [bt^{-1}t]$.
By definition there exist $u,v \in S$ such that
$u \leq t,at^{-1}t$ and $v \leq t,bt^{-1}t$ such that
$(t \setminus u), (at^{-1}t \setminus u), (t \setminus v), (bt^{-1}t \setminus v) \in I$.
Now $[t] = [u] = [at^{-1}t]$ and $[t] = [v] = [bt^{-1}t]$.
Since $u,v \leq t$ it follows that $u \sim v$ and so $u \wedge v$ exists by \cite[Lemma 1.4.11]{Law1}.
Clearly, $u \wedge v \leq a,b$.
In addition $[t] = [u \wedge v]$.
We have proved that $\nu$ is weakly meet preserving.

Conversely, let $\theta \colon S \rightarrow T$ be weakly meet preserving.
Put $I = \mbox{ker}(\theta)$.
We prove that $\theta (a) = \theta (b)$ if and only if $(a,b) \in \varepsilon_{I}$.
Suppose first that $(a,b) \in \varepsilon_{I}$.
Then by definition, there is an element $u \leq a,b$ such that $(a \setminus u), (b \setminus u) \in I$.
But then $a = (a \setminus u) \vee u$ and $b = (b \setminus u) \vee u$.
It follows that $\theta (a) = \theta (u) = \theta (b)$.
Conversely, suppose that $\theta (a) = \theta (b)$.
Put $t = \theta (a) = \theta (b)$.
Then by the definition of a weakly meet preserving map, there exists $c \leq a,b$ such that $t \leq \theta (c)$.
It follows that $\theta (a) = \theta (c) = \theta (b)$.
Thus $\theta (a \setminus c) = 0 = \theta (b \setminus c)$.
We have therefore proved that $(a \setminus c), (b \setminus c) \in I$ and so $(a,b) \in \varepsilon_{I}$.
\end{proof}

We now develop a refinement of non-commutative Stone duality, Theorem~\ref{them:ncsd}, by restricting the class of morphisms considered.
As a first step, we prove the following lemma.

\begin{lemma}\label{lem:bingo} Let $\theta \colon H \rightarrow G$ be coherent continuous covering functor between Boolean groupoids.
Suppose, in addition, that the image of $\theta$ is an invariant subspace of $G$ and that $\theta$ induces a homeomorphism between $H$ and this image.
Then $\theta^{-1} \colon \mathsf{KB}(G) \rightarrow \mathsf{KB}(H)$ is a surjective (and so proper) weakly meet preserving morphism.
\end{lemma}
\begin{proof} Since $\theta$ is injective, it induces an injective function between $\{h \in H \colon \mathbf{d}(h) = e\}$
and the set $\{g \in G \colon \mathbf{d}(g) = \theta (e)\}$.
Now let $g \in G$ be such that $\mathbf{d}(g) = \theta (e)$.
By assumption, $\theta (H)$ is an invariant subset of $G$.
Thus $g \in \theta (H)$.
It follows that there is an $h \in H$ such that $\theta (h) = g$.
In particular, $\theta (\mathbf{d}(h)) = \theta (e)$.
But $\theta$ is injective and so $\mathbf{d}(h) = e$.
We have therefore proved that $\theta$ is a covering functor.
It therefore only remains to prove that $\theta^{-1}$ is surjective.
Let $B \in \mathsf{KB}(H)$.
Since $\theta$ is a homeomorphism, we know that $\theta (B)$ is open in the image of $\theta$.
Thus there is an open subset $U$ of $G$ such that $\theta (B) = \mbox{im}(\theta) \cap U$.
However, $U$ is a union of compact-open partial bisections $A_{i}$ in $G$.
Thus $\theta (B) = \mbox{im}(\theta) \cap \left( \bigcup_{i \in I} A_{i} \right)$.
But $\theta (B)$ is compact and so 
$\theta (B) = \mbox{im}(\theta) \cap \left( \bigcup_{i=1}^{n} A_{i} \right)$
for some finite subset of the compact-open partial bisections $A_{i}$.
It follows that $B = \theta^{-1}(A_{1}) \cup \ldots \cup \theta^{-1}(A_{n})$.
In particular, the elements $\theta^{-1}(A_{i})$ and $\theta^{-1}(A_{j})$ are compatible when $i \neq j$.
We now apply Lemma~\ref{lem:irn}, to construct an element $A \in \mathsf{KB}(G)$ such that $\theta^{-1}(A) = B$.
\end{proof}

We now focus on the relationship between additive ideals of a Boolean inverse semigroup and appropriate structures in its Stone groupoid.
A good deal of the following result is proved in \cite{Lenz} but we give all the details for the sake of completeness.

\begin{lemma}\label{lem:cola} Let $S$ be a Boolean inverse semigroup.
There is a dual order isomorphism between the set of additive ideals of $S$ and the set of
closed invariant subspaces of $\mathsf{G}(S)_{o}$.
\end{lemma}
\begin{proof} We first show that there is an order isomorphism between the set of additive ideals of $S$
and the set of open invariant subsets of $\mathsf{G}(S)_{o}$.

Let $I$ be an additive ideal of $S$.
Define
$$\mathsf{O}(I) = \bigcup_{e \in \mathsf{E}(I)} V_{e}.$$
By construction, this is an open subset of $\mathsf{G}(S)_{o}$.
We prove that it is also invariant.
Let $A$ be an ultrafilter in $S$ such that $A^{-1} \cdot A \in \mathsf{O}(I)$.
Then there exists $a \in A$ and $e \in \mathsf{E}(I)$ such that $a^{-1}a \leq e$.
Because $I$ is an ideal, it follows that $a^{-1}a \in I$,
and so $a \in I$
from which we get that $aa^{-1} \in I$.
Thus $A \cdot A^{-1} \in \mathsf{O}(I)$, as required.

Let $U \subseteq \mathsf{G}(S)_{o}$ be an open invariant subset.
Observe first that the invariance of $U$ implies that $V_{s^{-1}s} \subseteq U$ if and only if $V_{ss^{-1}} \subseteq U$.
To see why, suppose that $V_{s^{-1}s} \subseteq U$.
We prove that $V_{ss^{-1}} \subseteq U$.
Let $A \in V_{ss^{-1}}$.
Since $ss^{-1} \in A$, and $A$ is also an inverse subsemigroup, we know from the theory of ultrafilters
that $B = (As)^{\uparrow}$ is a well-defined ultrafilter.
Observe that the ultrafilter $\mathbf{d}(B)$ contains the element $s^{-1}s$ so that $\mathbf{d}(B) \in U$.
But $U$ is an invariant subset and so $A = \mathbf{r}(B) \in U$, as claimed.
Define
$$\mathsf{I}(U) = \{s \in S \colon V_{s^{-1}s} \subseteq U \}.$$
It is routine to check that this is an additive ideal of $S$.

It is clear that both $\mathsf{O}$ and $\mathsf{I}$ preserve set inclusion.
It remains only to show that they are mutually inverse.
Let $I$ be an additive ideal of $S$.
Suppose that $s \in I$.
Then $s^{-1}s \in I$.
It follows that $s \in \mathsf{I} \mathsf{O} (I)$.
Suppose that $s \in \mathsf{I} \mathsf{O} (I)$.
Then $V_{s^{-1}s} \subseteq \bigcup_{e \in \mathsf{E}(S)} V_{e}$.
But $V_{s^{-1}s}$ is compact.
Thus there are a finite number of idempotents $e_{1}, \ldots, e_{m} \in I$ such that
$V_{s^{-1}s} \subseteq V_{e_{1}} \cup \ldots \cup V_{e_{m}} = V_{e_{1} \vee \ldots \vee e_{m}}$.
It follows that $s^{-1}s \leq e_{1} \vee \ldots \vee e_{m}$.
But $I$ is an additive ideal so that $e_{1} \vee \ldots \vee e_{m} \in I$
from which we get that $s^{-1}s \in I$ and so $s \in I$, as required.
We have therefore proved that $I = \mathsf{I} \mathsf{O}(I)$.
Now let $U$ be an open invariant subset of $\mathsf{G}(S)_{o}$.
Clearly, $\mathsf{O} \mathsf{I} (U) \subseteq U$.
To prove the reverse inclusion, let $x \in U$.
Since $U$ is an open set there is an idempotent $e \in S$ such that
$x \in V_{e} \subseteq U$, from the properties of the topology on $\mathsf{G}(S)$.
It follows that $e \in \mathsf{I}(U)$.
It is now immediate that $x \in \mathsf{O} \mathsf{I} (U)$.

To finish off, there is a dual order isomorphism between the set of open invariant subsets of $\mathsf{G}(S)_{o}$
and the set of closed invariant subsets of $\mathsf{G}(S)_{o}$ which is simply proved using set complementation with respect to $\mathsf{G}(S)_{o}$.\end{proof}

Let $G$ be a Boolean groupoid and let $X$ be a closed invariant subset of $G_{o}$.
Denote by $I_{X}$ the additive ideal in $\mathsf{KB}(G)$ associated with it as guaranteed by Lemma~\ref{lem:cola}.
The following explicit description of $I_{X}$ is immediate from the constructions and the definition of an invariant subset.

\begin{lemma}\label{lem:stuff}
Let $G$ be a Boolean groupoid and let $X$ be a closed invariant subset of $G_{o}$.
Then
$$A \in I_{X} 
\Longleftrightarrow A^{-1}A \cap X = \varnothing 
\Longleftrightarrow AA^{-1} \cap X = \varnothing
\Longleftrightarrow A \cap G_{X} = \varnothing.$$
\end{lemma}

The following result was stated, but not proved, at \cite[page~75]{Paterson}.

\begin{lemma}\label{lem:water} Let $G$ be a Boolean groupoid and let $X \subseteq G_{o}$ be a closed, invariant subset.
Then $G|_{X}$ is a Boolean groupoid with space of identities homeomorphic to $X$.
\end{lemma}
\begin{proof} By definition, $G_{o}$ is a Hausdorff space with a basis of compact-open sets.
Subspaces of Hausdorff spaces are Hausdorff.
Let $B$ be a a compact-open subset of $G_{o}$.
Then it is also closed.
It follows that $B \cap X$ is closed.
But $B \cap X \subseteq B$ and $B$ is a compact Hausdorff space.
It follows that $B \cap X$ is compact.
Thus $X$ is a Hausdorff space with a basis of compact-open subsets and so is a Boolean space.
It is now routine to check that $G|_{X}$ equipped with the subspace topology is an \'etale groupoid.
\end{proof}

The following lemma was communicated to us by Enrique Pardo with a proof suggested by Lisa Orloff Clark.

\begin{lemma}\label{lem:fruit} Let $G$ be a topological groupoid and let $X$ be a closed invariant subset of $G_{o}$. 
If $K \subseteq G$ is compact (in $G$), then $K \cap \mathbf{d}^{-1}(X)$ is compact in $G|_{X}$.
\end{lemma}
\begin{proof} The set $\mathbf{d}^{-1}(X)$ is closed in $G$
since $\mathbf{d} \colon G\rightarrow G_{o}$ is continuous,
and so $W = G \setminus \mathbf{d}^{-1}(X)$ is open.
Let $K\subseteq G$ be compact in $G$. 
Let 
$$K\cap \mathbf{d}^{-1}(X) \subseteq \bigcup_{i\in I} U_i $$
be an open covering in $G|_{X}$.
For each $i \in I$, there exists an open set $V_i \subseteq G$ such that $U_i = V_i \cap G|_{X}$. 
Thus,
$$K = (K \cap \mathbf{d}^{-1}(X)) \cup W \subseteq \bigcup_{i\in I} V_i \cup W.$$
Since $K$ is compact, there is a finite subcover
$$K \subseteq \bigcup_{i=1}^{n} V_i \cup W,$$   
from which we get that
$$K\cap \mathbf{d}^{-1}(X) \subseteq \bigcup_{i=1}^{n} U_i.$$
\end{proof}

We now assemble the above lemmas into the proof of a proposition.
Let $G$ be a Boolean groupoid and  $X$ be a closed invariant subset of $G_{o}$. 
Then $G|_{X}$ is a Boolean groupoid by Lemma~\ref{lem:water} and an invariant subgroupoid of $G$.
The embedding $G|_{X} \rightarrow G$ is coherent by Lemma~\ref{lem:fruit} and 
so this embedding is a coherent continuous covering functor.
By Lemma~\ref{lem:bingo}, there is, under non-commutative Stone duality, a surjective, weakly meet preserving morphism
$\theta \colon \mathsf{KB}(G) \rightarrow \mathsf{KB}(G|_{X})$
given by
$$\theta (A) = A \cap G|_{X} = A \cap \mathbf{d}^{-1}(X).$$
By Proposition~\ref{prop:anja}, this morphism is ideal-induced;
what that ideal should be is given by Lemma~\ref{lem:cola} and Lemma~\ref{lem:stuff}. 
We have therefore proved the following proposition;
this will deliver for us a proof of Theorem~\ref{them:four}.

\begin{proposition}\label{prop:earth} Let $G$ be a Boolean groupoid and $X$ a closed invariant subset of $G_{o}$.
Then $\mathsf{KB}(G|_{X}) \cong \mathsf{KB}(G)/I_{X}$.
\end{proposition}

\section{The Booleanization of an inverse semigroup}

In this section, we describe the structure of the Booleanization $\mathsf{B}(S)$ of the inverse semigroup $S$ 
described in detail in \cite{Lawson2018}.
This is the basis of Theorem~\ref{them:one}.
The following is well-known \cite[page 12]{Resende}.

\begin{proposition}\label{prop:lb} Let $G$ be a groupoid.
Then $\mathsf{L}(G)$, the set of all partial bisections of $G$ under subset multiplication, 
is a Boolean inverse semigroup in which the natural partial order is subset inclusion.
\end{proposition}
 
Let $S$ be an inverse semigroup.
Construct the groupoid  $\mathcal{L}(S)$ of proper filters of $S$ and then the Boolean inverse semigroup $\mathsf{L}(\mathcal{L}(S))$
of all partial bisections of $\mathcal{L}(S)$.
For each $a \in S$, define $U_{a}$ to be the set of all proper filters that contains $a$.
The following is proved in \cite{Lawson2018}.

\begin{lemma}\label{lem:hill} Let $S$ be an inverse semigroup.
\begin{enumerate}
\item $U_{0} = \varnothing$. 
\item $U_{a} = U_{b}$ if and only if $a = b$.
\item $U_{a}^{-1} = U_{a^{-1}}$.
\item $U_{a}U_{b} = U_{ab}$.
\item $U_{a}$ is a partial bisection.
\item $U_{a} \cap U_{b} = \bigcup_{x \leq a,b} U_{x}$.
\end{enumerate}
\end{lemma}

There is therefore an injective homomorphism $\upsilon \colon S \rightarrow \mathsf{L}(\mathcal{L}(S))$.
Let $a \in S$ and $a_{1}, \ldots, a_{m}$.
Define 
$$U_{a;a_{1}, \ldots, a_{m}} = U_{a} \cap U_{a_{1}}^{c} \cap \ldots \cap U_{a_{m}}^{c}.$$
Clearly, $U_{a:a_{1}, \ldots, a_{m}}$ is a partial bisection and so an element of $\mathsf{L}(\mathcal{L}(S))$.
The following is proved in \cite{Lawson2018}.

\begin{lemma}\label{lem:house} Let $S$ be an inverse semigroup.
\begin{enumerate}
\item $U_{a;a_{1}, \ldots, a_{m}}^{-1} = U_{a^{-1};a_{1}^{-1}, \ldots, a_{m}^{-1}}$.

\item $U_{a;a_{1}, \ldots, a_{m}} U_{b;b_{1}, \ldots, b_{n}} = U_{ab; ab_{1}, \ldots, ab_{n}, a_{1}b, \ldots, a_{m}b}$.
\end{enumerate}
\end{lemma}

With this preparation out of the way,
define $\mathsf{B}(S)$ to be that subset of $\mathsf{L}(\mathcal{L}(S))$ which consists of finite compatible unions
of elements of the form $U_{a;a_{1}, \ldots, a_{m}}$.
Define $\beta \colon S \rightarrow \mathsf{B}(S)$ by $s \mapsto U_{s}$.
Then this is the Booleanization of $S$ \cite{Lawson2018}. 
If $\theta \colon S \rightarrow T$ is a homomorphism to a Boolean inverse semigroup $T$
then there is a unique morphism $\phi \colon \mathsf{B}(S) \rightarrow T$ given by 
$\phi (U_{a;a_{1}, \ldots, a_{m}}) = \theta (a) \setminus (\theta (a_{1}) \vee \ldots \vee \theta (a_{m}))$
such that $\phi \beta = \theta$.
For later reference, the topology defined on the groupoid of proper filters of $S$ using the sets of the form $U_{a;a_{1}, \ldots, a_{m}}$
is called the {\em patch topology}.\\

\noindent
{\bf Terminology. }What we call the `patch topology', this is the term used by Johnstone \cite{J}, is identical to the topology inherited from the product topology and
to what is also termed the topology of pointwise convergence (see \cite[page 174]{Paterson}).
Thus the topologies used in this paper, in \cite{Exel} and in \cite{Paterson} are identical.\\

\section{The Exel completion: proof of Theorem~\ref{them:three}}

We can now prove our first main new theorem.
The proof we shall give will be based on Section~6.
The notions of cover and cover-to-join map defined in the Introduction are central.
Let $S$ be an inverse semigroup.
From Section~6, we shall need the description of the Booleanization $\mathsf{B}(S)$.
Define $I$ to be the closure under finite compatible joins of all elements $U_{a;a_{1}, \ldots,a_{m}}$ of $\mathsf{B}(S)$ 
where $\{a_{1}, \ldots, a_{m} \} \rightarrow a$.

\begin{lemma}\label{lem:tea} The set $I$ is an additive ideal of $\mathsf{B}(S)$. 
\end{lemma}
\begin{proof} By symmetry, it is enough to prove that if $U_{a;a_{1}, \ldots,a_{m}}$ is such that 
$$\{a_{1}, \ldots, a_{m} \} \rightarrow a$$
and $U_{b;b_{1}, \ldots, b_{n}}$ is any element then $U_{a;a_{1}, \ldots,a_{m}}U_{b;b_{1}, \ldots, b_{n}} \in I$.
By Lemma~\ref{lem:house}, we have that
$$U_{a;a_{1}, \ldots,a_{m}}U_{b;b_{1}, \ldots, b_{n}} = U_{ab;ab_{1}, \ldots, ab_{n}, a_{1}b, \ldots, a_{m}b}.$$
We prove that $\{ab_{1}, \ldots, ab_{n}, a_{1}b, \ldots, a_{m}b\} \rightarrow ab$.
Let $0 < x \leq ab$.
Then $xb^{-1}b = x$ and so, in particular, $xb^{-1} \neq 0$.
Thus $0 \neq xb^{-1} \leq abb^{-1} \leq a$.
It follows that there is $0 \neq y \leq xb^{-1}, a_{i}$ for some $i$.
In particular, $y = ybb^{-1}$ and so $yb \neq 0$.
Hence $0 \neq yb \leq x, a_{i}b$.
\end{proof}

By Lemma~\ref{lem:tea}, we may therefore form the quotient Boolean inverse semigroup $\mathsf{B}(S)/I = \mathsf{B}(S)/\varepsilon_{I}$.
Denote the elements of $\mathsf{B}(S)/I$ as elements of $\mathsf{B}(S)$ enclosed in square brackets.
Denote by $\nu \colon \mathsf{B}(S) \rightarrow \mathsf{B}(S)/I$ the natural morphism.
Put $\mathsf{B}(S)/I = \mathsf{T}(S)$, a Boolean inverse semigroup of course, and $\tau = \nu \beta$.
We prove that $\tau \colon S \rightarrow \mathsf{T}(S)$ is universal for cover-to-join maps from $S$ to Boolean inverse semigroups.
To do this, observe that the operations in $\mathsf{B}(S)$ are set-theoretic.
It follows that if $a_{1}, \ldots, a_{m} \leq a$ then
$$U_{a;a_{1}, \ldots, a_{m}} = U_{a} \setminus \left( U_{a_{1}} \cup \ldots \cup U_{a_{m}} \right).$$
The natural map $\nu$ is a morphism of Boolean inverse semigroups and so we have that
$$[U_{a;a_{1}, \ldots, a_{m}}] = [U_{a}] \setminus \left( [U_{a_{1}}] \cup \ldots \cup [U_{a_{m}}] \right).$$
We prove first that $\tau$ is itself a cover-to-join map.
Suppose that $\{a_{1}, \ldots, a_{m} \} \rightarrow a$.
Then, by definition $[U_{a;a_{1}, \ldots, a_{m}}] = 0$.
It follows that $[U_{a}] = [U_{a_{1}}] \vee \ldots \vee [U_{a_{m}}]$. 
Next, let $\theta \colon S \rightarrow T$ be a cover-to-join map where $T$ is Boolean.
Then by Theorem~\ref{them:one} and Section~6, the Booleanization theorem, there is a unique morphism of Boolean inverse semigroups $\phi \colon \mathsf{B}(S) \rightarrow T$
such that $\phi \beta = \theta$ and given by $\phi (U_{a:a_{1}, \ldots,a_{m}}) = \theta (a) \setminus (\theta (a_{1}) \vee \ldots \vee \theta (a_{m}))$.
However, $\phi$ is a cover-to-join map and so 
if $\{a_{1}, \ldots, a_{m} \} \rightarrow a$ then $\phi (U_{a;a_{1}, \ldots,a_{m}}) = 0$.
Clearly, $I \subseteq \mbox{ker}(\phi)$. 
Thus there is a unique morphism $\psi \colon \mathsf{B}(S)/I \rightarrow T$ such that $\psi \nu = \phi$.
We therefore have that $\psi \tau = \theta$.
It remains to show that $\psi \colon \mathsf{T}(S) \rightarrow T$ is the unique
morphism such that $\psi \tau = \theta$.
Observe that any morphism $\psi'$ such that  $\psi' \tau = \theta$ must map $[U_{a}]$ to $\theta (a)$.
The result now follows by observing that $\psi'$ is a morphism and so is a morphism of unital Boolean algebras
when restricted to the principal order ideal generated by $[U_{a}]$.
It follows that $\psi' ([U_{a;a_{1},\ldots, a_{m}}]) = \theta (a) \setminus (\theta (a_{1}) \vee \ldots \vee \theta (a_{m}))$.\\

This concludes the proof of Theorem~\ref{them:three}.

\section{Tight filters}

The material in this section is due to Exel \cite{Exel} with some ideas from \cite{LL}.
We begin with some well-known results on ultrafilters.
The following is proved using the same ideas as in \cite[Proposition~2.13]{Law3}.

\begin{lemma}\label{lem:lettuce} Let $S$ be an inverse semigroup
and let $A$ be a proper filter in $S$.
Then the following are equivalent:
\begin{enumerate}

\item $A$ is an ultrafilter.

\item $\mathbf{d}(A)$ is an ultrafilter.

\item $\mathbf{r}(A)$ is an ultrafilter

\end{enumerate}
\end{lemma}

Likewise, the following is proved using the same ideas as in \cite[Proposition~2.13]{Law3}.

\begin{lemma}\label{lem:carrot} Let $S$ be an inverse semigroup.
Then there is a bijection between the set of idempotent ultrafilters in $S$
and the set of ultrafilters in the meet-semilattice $\mathsf{E}(S)$.
In particular, the bijection is given by the following two maps: 
if $A$ is an idempotent ultrafilter in $S$ then $A \cap \mathsf{E}(S)$ is an ultrafilter in
$\mathsf{E}(S)$; 
if $F$ is an ultrafilter in $\mathsf{E}(S)$ then $F^{\uparrow}$ is an idempotent ultrafilter in $S$. 
\end{lemma}

The following is a simple consequence of Zorn's lemma.

\begin{lemma}\label{lem:hedgehog} Let $S$ be an inverse semigroup.
Then each non-zero element of $S$ is contained in an ultrafilter.
\end{lemma}

A very useful result in working with ultrafilters is the following \cite[Lemma~12.3]{Exel}.

\begin{lemma}\label{lem:exel} Let $E$ be a meet semilattice with zero.
A proper filter $A$ in $E$ is an ultrafilter if and only if 
$e \in E$ such that $e \wedge a \neq 0$ for all $a \in A$ implies that $e \in A$.
\end{lemma}

Let $S$ be an arbitrary inverse semigroup.
Associated with $S$ is its Booleanization $\mathsf{B}(S)$.
The Stone groupoid of $\mathsf{B}(S)$ is Paterson's universal groupoid $\mathsf{G}_{u}(S)$
which consists of the groupoid of proper filters of $S$ equipped with the patch topology.\\

\noindent
{\bf Definition. }The space of identities of  $\mathsf{G}_{u}(S)$ is denoted by $\mathsf{X}(S)$.
It is simply the set of all proper filters of $\mathsf{E}(S)$ equipped with the patch topology.\\

\noindent
{\bf Definition. }The {\em Cuntz-Krieger boundary of $S$}, denoted by $\partial S$, is the closure of the set
of ultrafilters in  $\mathsf{X}(S)$.\\ 

We shall now characterize the elements of $\partial S$ in algebraic terms.
A proper filter $A$ of $S$ (we reiterate that $S$ is an inverse semigroup, we do not assume that it is a monoid) is said to be {\em tight} if $a \in A$
and $C \rightarrow a$ implies that $C \cap A \neq \varnothing$.

\begin{remark}{\em The reader is alerted to the fact that our use of the word `tight' is a slight restriction of
the way it is used in \cite{Exel}. The salient point is that Exel wishes to work in an environment where he can be neutral as to whether his
semigroups have an identity or not. In addition, he only works with unital Boolean algebras (in our terminology).
Nevetheless, Exel's tight groupoid and ours are the same.}
\end{remark}

\begin{remark}
{\em To provide some further context: the relationship between covers and tight filters is analogous to the relationship
between joins and prime filters.}
\end{remark}

The following result was first proved in \cite{Exel} where the closure of the set of ultrafilters was characterized in terms of tight filters;
it is also implicit in the work of \cite{Lenz} but there conditions are sought to ensure that the set of ultrafilters is already closed.
 
\begin{lemma}\label{lem:pink} Let $S$ be an inverse semigroup (we reiterate, that we do not assume that $S$ is a monoid).
\begin{enumerate}

\item Every ultrafilter in $\mathsf{E}(S)$ is tight.

\item Every open set containing a tight filter contains an ultrafilter.

\item The set of tight filters in $\mathsf{E}(S)$ is a closed subspace of $\mathsf{X}(S)$.

\item The set of tight filters in $\mathsf{E}(S)$ is the closure in  $\mathsf{X}(S)$ of the set of ultrafilters.

\end{enumerate}
\end{lemma}
\begin{proof} (1) Let $A$ be an ultrafilter.
Suppose that it is not tight.
Then there is an element $a \in A$ and a cover $C \rightarrow a$ such that $C \cap A = \varnothing$;
that is, no element of $C$ belongs to $A$.
It follows by Lemma~\ref{lem:exel}, 
that for each $c_{i} \in C$, there is $a_{i} \in A$ such that $c_{i} \wedge a_{i} = 0$.
Since $a_{1}, \ldots, a_{m} \in A$ it follows that $e = a_{1} \wedge \ldots \wedge a_{m} \in A$.
Now, also, $a \in A$ and so $a \wedge e \neq 0$.
In particular, $a \wedge e \leq a$.
It follows that $c_{i} \wedge a \wedge e \neq 0$ for some $c_{i}$.
But $c_{i} \wedge e = 0$, which is a contradiction.

(2) Let $A$ be a tight filter. 
We prove that every open set containing $A$ contains an ultrafilter.
Let $A \in U_{a:a_{1},\ldots,a_{m}}$.
Since $A$ is tight, it cannot be that $\{a_{1},\ldots,a_{m}\}$ is a cover of $a$.
Thus there is a non-zero element $x \leq a$ such that $x \wedge a_{i} = 0$ for $1 \leq i \leq m$.
By Lemma~\ref{lem:hedgehog}, let $F$ be an ultrafilter that contains $x$.
Then it clearly cannot contain any of the elements $a_{1}, \ldots, a_{m}$.
We have therefore proved that $F \in U_{a:a_{1},\ldots,a_{m}}$.

(3) Let $A$ be an element of $\mathsf{X}(S)$ with the property that every open set containing $A$ contains a tight filter.
We prove that $A$ is also a tight filter.
Suppose not.
Then there is an element $a \in A$ and a cover $C = \{c_{1}, \ldots, c_{m}\} \rightarrow a$ such that $A \cap C = \varnothing$.
It follows that $A \in U_{a; c_{1}, \ldots, c_{m}}$.
However, the open set $U_{a; c_{1}, \ldots, c_{m}}$ contains no tight filters (since it is not possible for a tight filter to contain $a$ but omit all the elements $c_{1}, \ldots, c_{m}$) 
but does contain $A$, which contradicts our assumption on $A$.

(4) Let $A$ be a filter such that every open set containing $A$ contains an ultrafilter.
Then, by part (1), it is certainly the case that every open set containing $A$ contains a tight filter.
It follows by part (3), that $S$ is itself a tight filter.
\end{proof}

The following is proved as \cite[Lemma 5.9]{LL}.

\begin{lemma}\label{lem:crystal} Let $S$ be an inverse semigroup
and let $A$ be a proper filter in $S$.
Then the following are equivalent:
\begin{enumerate}

\item $A$ is a tight filter.

\item $\mathbf{d}(A)$ is a tight filter.

\item $\mathbf{r}(A)$ is a tight filter.
\end{enumerate}
\end{lemma}

The following is now immediate.

\begin{corollary}\label{cor:quisling} 
The Cuntz-Krieger boundary is a closed, invariant subspace of the space of identities of the universal groupoid.
\end{corollary}

\begin{remark}{\em Exel's definition of a tight character \cite[page 54]{Exel} and our definition of a tight filter are two ways of looking at the same class of objects.
The explanation for these different characterizations simply boils down to the nature of the basis that one chooses to work with;
Exel's is more generous and ours more parsimonious.
In our filter setting, Exel's basic open sets have the form $U_{X,Y}$ where $X$ and $Y$ are finite sets
and $U_{X,Y}$ is defined to be those proper filters that contain all of the elements of $X$ but omit all of the elements of $Y$.
When $X$ is non-empty, it is easy to show that $U_{X,Y}$ is equal to a set of the form $U_{a;a_{1},\ldots, a_{n}}$ for some $a$ and subset $\{a_{1},\ldots, a_{n}\} \subseteq a^{\downarrow}$.
When $X$ is empty, we have that $U_{\varnothing, Y} = \overline{\bigcup_{e \in Y} U_{e}}$;
observe that the sets $U_{e}$ are compact in the (Hausdorff) patch topology and so closed.
We now use the fact that the sets of the form $U_{a;a_{1}, \ldots, a_{n}}$ form a basis for the patch topology.}
\end{remark}

\section{The Stone groupoid of the Exel completion: proof of Theorem~\ref{them:four}}

We can now prove our second main result.
Let $S$ be an inverse semigroup.
By Corollary~\ref{cor:quisling} and Lemma~\ref{lem:water}, it follows that the reduction $\mathsf{G}_{u}(S)|_{\partial S}$ is a Boolean groupoid;
it is the {\em tight groupoid} of $S$ \cite{Exel}
and can simply be regarded as the groupoid of tight filters with the restriction of the patch topology.
We denote this groupoid by $\mathsf{G}_{t}(S)$.
We call the associated Boolean inverse semigroup $\mathsf{KB}(\mathsf{G}_{t}(S))$ the {\em tight semigroup} of $S$.
There is a map from $S$ to $\mathsf{KB}(\mathsf{G}_{t}(S))$, which we shall denote by $\eta$, which takes $a$ 
to the set of tight filters containing $a$, a set we shall denote by $U_{a}^{t}$.
By Lemma~\ref{lem:pink} and Lemma~\ref{lem:hedgehog},
$a \neq 0$ implies that $U_{a}^{t} \neq \varnothing$.

\begin{lemma}\label{lem:red} 
The map $\eta$ is a cover-to-join map.
\end{lemma}
\begin{proof} We begin with an observation
Let $a_{1}, \ldots, a_{m} \leq a$.
Then $U_{a;a_{1},\ldots,a_{m}} \cap \mathsf{G}_{t}(S) = \emptyset$ if and only if $\{a_{1}, \ldots, a_{m}\} \rightarrow a$.
Suppose first that  $\{a_{1}, \ldots, a_{m}\} \rightarrow a$ then any tight filter containing $a$ must contain at least one of the $a_{i}$,
for some $i$. It follows that $U_{a;a_{1},\ldots,a_{m}} \cap \mathsf{G}_{t}(S) = \emptyset$.
Conversely, let $U_{a;a_{1},\ldots,a_{m}} \cap \mathsf{G}_{t}(S) = \emptyset$.
Suppose that $\{a_{1},\ldots,a_{m}\}$ is not a cover of $a$.
Then there is some $0 \neq x \leq a$ such that $x \wedge a_{i} = 0$ for all $1 \leq i \leq m$.
By Lemma~\ref{lem:hedgehog}, there is an ultrafilter $A$ containing $x$.
But, clearly, $a_{i} \notin A$ for all $1 \leq i \leq m$. 
Thus $A \in U_{a;a_{1}, \ldots, a_{m}}$.
But ultrafilters are tight filters by Lemma~\ref{lem:pink}.
This contradicts our assumption that 
$U_{a;a_{1},\ldots,a_{m}} \cap \mathsf{G}_{t}(S) = \emptyset$.
It follows that $\{a_{1}, \ldots,a_{m}\} \rightarrow a$.

Let $\{a_{1}, \ldots, a_{m}\} \rightarrow a$.
Then $\eta (a_{1}) \vee \ldots \vee \eta (a_{m}) \leq \eta (a)$.
Suppose that the inequality were strict.
Then there would be a tight filter containing $a$ that omitted $a_{1}, \ldots, a_{m}$
but this is impossible by the first part of the proof.
It follows that $\eta (a) = \bigvee_{i=1}^{m} \eta (a_{i})$.
\end{proof}

We shall now prove that the Stone groupoid of the Exel completion is the tight groupoid.
Recall that by Theorem~\ref{them:two},  
$\mathsf{G}(\mathsf{B}(S))$ is just the universal groupoid $\mathsf{G}_{u}(S)$.
By Corollary~\ref{cor:quisling}, $\partial S$ is a closed invariant subspace of the space of identities of $\mathsf{G}_{u}(S)$.
Thus by Proposition~\ref{prop:earth}, we have the following isomorphism of Boolean inverse semigroups:
$$\mathsf{KB} (\mathsf{G}_{u}(S)|_{\partial S}) \cong \mathsf{KB}(\mathsf{G}_{u}(S)) / I_{\partial S}.$$
By definition, 
$\mathsf{G}_{u}(S)|_{\partial S} = \mathsf{G}_{t}(S)$
is the tight groupoid.
By Theorem~\ref{them:two}, the Boolean inverse semigroup $\mathsf{KB}(\mathsf{G}_{u}(S))$ is just $\mathsf{B}(S)$, the Booleanization of $S$.
We therefore have that
$$\mathsf{KB} (\mathsf{G}_{t}(S)) \cong \mathsf{B}(S) / I_{\partial S}.$$
It therefore remains to identify the elements of the additive ideal $I_{\partial S}$.
To do this, it is enough to identify the elements of the form $U_{a;a_{1}, \ldots,a_{m}}$ which belong to $I_{\partial S}$.
However, from the definitions, $U_{a:a_{1}, \ldots, a_{m}}^{t} = \varnothing$ if and only if $\{a_{1}, \ldots, a_{m}\} \rightarrow a$.
Thus the elements of the form $U_{a;a_{1}, \ldots,a_{m}}$ which belong to $I_{\partial S}$ are precisely those for which  $\{a_{1}, \ldots, a_{m}\} \rightarrow a$.
We have therefore proved that $\mathsf{B}(S) / I_{\partial S} = \mathsf{T}(S)$, the Boolean inverse semigroup described in Section~7.
It is now immediate by Theorem~\ref{them:ncsd}, that the Stone groupoid of the Exel completion of $S$ is the tight groupoid.\\

This concludes the proof of Theorem~\ref{them:four}.

\section{Tiling semigroups}

Kellendonk associated inverse semigroups with (aperiodic) tilings and then showed how to construct \'etale groupoids and $C^{\ast}$-algebras from them \cite{Kellendonk1, Kellendonk2}.
The construction of the inverse semigroups was formalized in \cite{KL}
and the construction of the \'etale groupoid from the inverse semigroup was described in \cite{Lenz}.
Within the framework of this paper, inverse semigroups were being considered in which the tight filters were the ultrafilters.
Meet semilattices with this property were termed {\em compactable} in \cite{Lawson2010} where they were characterized \cite[Theorem 2.10]{Lawson2010} in
terms introduced by \cite{Lenz}.
A more concrete sufficient condition was formulated as \cite[Proposition 2.14]{Lawson2010}.
This theme was taken up in a more general frame in \cite{Law5} where an inverse semigroup was termed {\em pre-Boolean} if
every tight filter was an ultrafilter. 
Neither of the terms `compactable' or `pre-Boolean' is satisfactory but these examples show that
a single term is needed to signify that all tight filters are ultrafilters;
the term {\em finitely complex} is a possibility.
Both papers \cite{EGS} and \cite{LL} focus on the inverse semigroups constructed from tilings
and the conditions on the tiling that force the tight filters to be ultrafilters.


\section{Abstract and concrete Cuntz-Krieger relations}

We may summarize what we have found in this paper as follows.
Let $S$ be an inverse semigroup
and 
let $\{a_{1}, \ldots, a_{m} \} \rightarrow a$ be a cover of the element $a$ in $S$.
Then this gives rise to a relation $a = \bigvee_{i=1}^{m} a_{i}$ in the Booleanization $\mathsf{B}(S)$ of $S$
(with an appropriate abuse of notation).
When $\mathsf{B}(S)$ is factored out by all such relations,
we have proved that we get the Exel completion $\mathsf{T}(S)$ of $S$;
in this case, its Stone groupoid is precisely Exel's tight groupoid $\mathsf{G}_{t}(S)$.
In this paper, we treat the relations of the form  $a = \bigvee_{i=1}^{m} a_{i}$  as Cuntz-Krieger relations
--- let us call them {\em abstract Cuntz-Krieger relations}.
It is natural to ask what evidence there is for this terminology.
Of course, Cuntz-Krieger relations are defined in rather concrete situations so to justify our claim,
it is enough to check that in those concrete situations, the abstract Cuntz-Krieger relations above give all and only
the concrete Cuntz-Krieger relations.
First of all, we may restrict our attention to relations involving only idempotents.
The following is proved as \cite[lemma~3.1(1)]{Law5};
it is a consequence of the fact that the principal order ideals $a^{\downarrow}$ and $\mathbf{d}(a)^{\downarrow}$ are
order isomorphic.

\begin{lemma}\label{lem:idpt} Let $S$ be an inverse semigroup and let $\theta \colon S \rightarrow T$ be a homomorphism 
to a Boolean inverse semigroup.
Then $\theta \colon S \rightarrow T$ is a cover-to-join map if and only if $\theta \colon \mathsf{E}(S) \rightarrow \mathsf{E}(T)$
is a cover-to-join map.
\end{lemma}

Next, we may focus on those relations determined by certain distinguished idempotents.
The following is proved as \cite[lemma~3.1(2)]{Law5}.

\begin{lemma}\label{lem:d} Let $S$ be an inverse semigroup and let $\{e_{i} \colon i \in I\}$ be an idempotent transversal of the of the
non-zero $\mathscr{D}$-classes. 
Let $\theta \colon S \rightarrow T$ be a homomorphism to a Boolean inverse semigroup.
Then $\theta$ is a cover-to-join map if and only if it is a cover-to-join map for the distinguished family of idempotents.
\end{lemma}

\begin{example}{\em Our first example goes right back to the origin of the Cuntz-Krieger relations
and Cuntz's original paper \cite{Cuntz}.
We shall treat everything in the context of (Boolean) inverse semigroups.
An inverse semigroup $S$ is said to be {\em $0$-bisimple} if it has exactly one non-zero $\mathscr{D}$-class.
Let $S$ be a Boolean inverse monoid.
Then in the light of Lemma~\ref{lem:idpt} and Lemma~\ref{lem:d}, we can focus entirely on the covers of the identity.
An inverse semigroup is said to be $E^{\ast}$-unitary if $0 \neq e \leq a$, where $e$ is an idempotent, implies that $a$ is an idempotent.
In \cite[Remark 2.3]{Lenz}, it is proved that an $E^{\ast}$-unitary inverse semigroup is a $\wedge$-semigroup.
The most important class of examples of $E^{\ast}$-unitary, $0$-bisimple inverse monoids are the polycyclic inverse monoids $P_{n}$ ($n \geq 2$).
Recall that 
$$P_{n} = \langle a_{1}, \ldots, a_{n} \colon a_{i}^{-1}a_{i} = 1, a_{i}^{-1}a_{j} = 0 \rangle.$$
The Exel completion of $P_{n}$, denoted by $C_{n}$ and called the {\em Cuntz inverse monoid}, was constructed in \cite{Law2007, Law2007b}, developing aspects of \cite{Birget}, and
was further studied in \cite{LS}.
Representations of $C_{n}$ by certain kinds of partial bijections were constructed in \cite{JL}, based on the work in \cite{BJ},
and subsequently extended in \cite{HW}.
We need only focus on the covers of the identity.
An immediate example is the cover $\{a_{1}a_{1}^{-1}, \ldots, a_{n}a_{n}^{-1}\} \rightarrow 1$.
Observe that $\{a_{1}, \ldots, a_{n}\}$ is a maximal prefix code in the free monoid $A_{n}^{\ast} = \{a_{1}, \ldots, a_{n}\}^{\ast}$.
In fact, the covers of $1$ are in bijective correspondence with the maximal prefix codes of $A_{n}^{\ast}$.
The following is immediate from \cite[Section 4.1]{LL}; recall that in an inverse semigroup if $e$ is an idempotent then $aea^{-1}$ is an idempotent:
let $S$ be an inverse semigroup.
If $\{e_{1}, \ldots, e_{m} \} \rightarrow e$, where $e$ is an idempotent, and $a$ is any element,
then either $aea^{-1} = 0$ or $\{ae_{1}a^{-1}, \ldots, ae_{m}a^{-1} \} \rightarrow aea^{-1}$.
By this result and \cite[Proposition~II.4.7]{BP}, we therefore have the following:
the Cuntz-Krieger ideal of $\mathsf{B}(P_{n})$ is generated by 
$$1 \setminus (a_{1}a_{1}^{-1} \vee \ldots \vee a_{n}a_{n}^{-1}).$$
We may therefore regard the Cuntz inverse monoid as being the quotient of the Booleanization $\mathsf{B}(P_{n})$ factored
out by the relation given by $1 = a_{1}a_{1}^{-1} \vee \ldots \vee a_{n}a_{n}^{-1}$.}
\end{example}

\begin{example}{\em The Cuntz inverse monoids can be generalized to what we call then {\em Cuntz-Krieger monoids}, $CK_{G}$, where $G$ is a finite graph \cite{JL2014}.
Thus we now consider the paper \cite{CK} from our perspective.
From a (finite) directed graph $G$, one constructs a free category and from that, in a manner reminiscent of the way in which the polycyclic
inverse monoids are constructed from free monoids, one constructs the so-called {\em graph inverse semigroups} $P_{G}$.
The Exel completion of $P_{G}$ is called the {\em Cuntz-Krieger semigroup}, $CK_{G}$.
In \cite[Theorem~2.1]{JL2014} an abstract characterization of graph inverse semigroups is given.
In particular, each non-zero $\mathscr{D}$-class has a unique maximal idempotent.
We may therefore restrict attention to covers of maximal idempotents.
If the graph $G$ has the property that the in-degree of each vertex is finite, then each maximal idempotent $e$ is {\em pseudofinite}
defined as follows: denote by $\hat{e}$ the set of all idempotents $f$ such that $f < e$ and $e$ covers $f$;
the idempotents in $\hat{e}$ are therefore those immediately below $e$;
we assume that $\hat{e}$ is finite and that if $g < e$ then $g \leq f < e$ for some $f \in \hat{e}$.
It follows that for each maximal idempotent, we have that $\hat{e} \rightarrow e$.
The inverse semigroups $P_{G}$ are $E^{\ast}$-unitary (and so are $\wedge$-semigroups) and their semilattices of idempotents are {\em unambiguous}
which means that if $0 \neq e \leq i,j$, where $e,i,j$ are all idempotents, then $i \leq j$ or $j \leq i$.
This implies that we can restrict attention to covers that consist of orthogonal elements (as in the case of maximal prefix codes in free monoids)
\cite[Corollary]{JL2014}. By an argument analogous to the one used in \cite[Lemma~3.9]{JL2014},
the Cuntz-Krieger ideal of $\mathsf{B}(P_{G})$ is generated by elements of the form
$$e \setminus \left( \bigvee_{f \in \hat{e}} f \right)$$
where $e$ is a maximal idempotent in $P_{G}$.}
\end{example}

The two examples above show that what we term `abstract Cuntz-Krieger relations' do agree with the concrete 
Cuntz-Krieger relations at least for suitably nice Cuntz-Krieger algebras.
The most general class of structures for which concrete Cuntz-Krieger relations have been introduced are the higher-rank graphs \cite{FMY, KP, RSY, RS}.
The relationship between what we term `abstract Cuntz-Krieger relations' and `concrete Cuntz-Krieger relations' was the subject of \cite{DM}
and served as one of the inspirations for our work.
The authors there prove a theorem, (\cite[Theorem~3.7]{DM}), which in our terminology states that for the inverse semigroups
arising as the inverse semigroups of zigzags in the countable, finitely aligned categories of paths of Spielberg \cite{Sp1}
the abstract and concrete Cuntz-Krieger relations coincide.
This result therefore applies in particular to finitely aligned higher-rank graphs.

\begin{remark}{\em It is worth noting that the Introduction to Spielberg's paper \cite{Sp1}
focuses on the nature of the concrete Cuntz-Krieger relations.
In addition, it also highlights the nature of the boundary which we have termed the `Cuntz-Krieger boundary'.}
\end{remark}

\nocite{*}


\end{document}